\documentclass[11pt]{article}
\usepackage{amssymb,amsmath,MnSymbol}
\usepackage{color}
\usepackage{amsthm}
\usepackage{graphicx}
\usepackage{color}
\usepackage{authblk}
\usepackage{url}
\usepackage[dvipsnames]{xcolor}
\usepackage{rotating,booktabs}
\usepackage{arydshln}
\usepackage{bbold}
\usepackage{enumerate}
 \usepackage[pagewise]{lineno}

\def\Sym{\mbox{\rm Sym}}

\renewcommand{\phi}{\varphi}
\newcommand{\F}{\mathbb{F}}
\newcommand{\Z}{\mathbb{Z}}
\newcommand{\Span}[1]{\left\langle\, #1 \,\right\rangle}

\newcommand{\AGL}{\mbox{\rm AGL}}

\newcommand{\veps}{\varepsilon}
\newcommand{\hex}[1]{{\texttt{#1}_{\text{x}}}}
\newcommand\deq{\mathrel{\stackrel{\makebox[0pt]{\mbox{\normalfont\tiny def}}}{=}}}

\newcommand{\Set}[1]{\left\{ #1 \right\}}

\newcommand{\K}[0]{\mathcal{K}}
\newcommand{\Hc}[0]{\mathcal{H}}

\usepackage{lipsum}

\usepackage[usenames,dvipsnames]{pstricks}
\usepackage{epsfig}
\usepackage{pst-grad} 
\usepackage{pst-plot} 

\usepackage{pstricks, pst-node, graphics}
\newpsobject{showgrid}{psgrid}{subgriddiv=1,griddots=10,gridlabels=6pt}

\makeatletter
\def\blfootnote{\gdef\@thefnmark{}\@footnotetext}
\makeatother

\DeclareMathOperator{\sym}{Sym}

\newtheorem{theorem}{Theorem}[section]
\newtheorem{definition}[theorem]{Definition}
\newtheorem{lemma}[theorem]{Lemma}
\newtheorem{proposition}[theorem]{Proposition}
\newtheorem{corollary}[theorem]{Corollary}
\newtheorem{remark}[theorem]{Remark}
\newtheorem{example}[theorem]{Example}

\title{
Type-Preserving Matrices and Security of Block Ciphers\blfootnote{\;\\\textit{Email addresses}: \texttt{ric.aragona@gmail.com} (R. Aragona), \texttt{alessio.meneghetti@unitn.it } (A. Meneghetti)}
}

\author[1]{Riccardo Aragona
\thanks{The first author is member of of INdAM-GNSAGA (Italy) and he thankfully acknowledges support by DISIM of the University of L'Aquila and by MIUR-Italy via PRIN 2015TW9LSR ``Group theory and applications''}
}
\author[2]{Alessio Meneghetti
}
\affil[1]{
\small{DISIM, Universit\`a degli Studi dell'Aquila 

Via Vetoio, 67100 Coppito (AQ), Italy
}}

\affil[2]{
\small{Dipartimento di Matematica, Universit\`a degli Studi di Trento 

Via Sommarive 14, 38123 Povo (TN),  Italy
}}

\date{}
\begin{document}
\maketitle
\vspace{-5mm}
\begin{abstract}
\noindent We introduce a new property for mixing layers which guarantees protection against algebraic attacks based on the imprimitivity of the group generated by the round functions. Mixing layers satisfying this property are called \textit{non-type-preserving}. 
Our main result is to characterize such mixing layers by providing a list of necessary and sufficient conditions on the structure of their underlying binary matrices. 
Then we show how several
families of linear maps are non-type-preserving, including the mixing layers of AES, GOST and PRESENT. Finally we prove that the group generated by the round functions of an SPN cipher with addition modulo $2^n$ as key mixing function is primitive if its mixing layer satisfies this property. 

\end{abstract}
\medskip
\noindent\small{\textbf{Keywords:}
Cryptosystems, mixing layer, group generated by the round functions, primitive groups}\\
\medskip
\small{\textbf{MSC 2010:} 20B15,
20B35,
94A60
}


\section{Introduction}\label{sec:intro}

Most modern block ciphers are iterated block ciphers, i.e. are obtained as composition of round functions, and belong to two families of cryptosystems, i.e. Substitution Permutation Networks (SPN) and Feistel Networks (FN). Within each round  three permutations of the plaintext space operate, i.e. a non-linear layer and a linear layer which respectively perform confusion and diffusion (see \cite{sha49}) and a key mixing function which combines the message with the corresponding round key.
Most SPN's use the XOR as key mixing, but in many Feistel Networks (e.g.  MARS \cite{mars}, GOST \cite{GOST}, RC6 \cite{RC6}, SEA \cite{sea}) and in other block ciphers not belonging to these two families (e.g. IDEA \cite{idea}) the key mixing function is the addition modulo $2^n$, for some  integer $n$.

Using addition modulo $2^n$ as key mixing function may increase the nonlinearity of a round function. Intuitively, one could take that adding an extra nonlinear layer increases the complexity of  attacks. Actually in \cite{lin-att} the authors prove from a theoretical point of view  that adopting a key mixing defined by an addition modulo $2^n$ can help to prevent linear cryptanalysis. Then they consider two toy SPN's, GPig1 and GPig2, with the same structure but  with key mixing respectively defined by XOR and addition modulo $2^n$ and check from an experimental point of view that the first one is weaker than the latter against linear cryptanalysis. On the other hand in \cite{alg-att} the authors investigate how the use of addition modulo $2^n$ in round functions influences
algebraic attacks. Also in \cite{prop-mod}  statistical and algebraic properties of  addition modulo a power of two are studied from a cryptographic point of view.


In this paper, we aim to investigate which properties of the mixing layer are useful to avoid particular classes of algebraic vulnerabilities on an SPN which uses  addition modulo $2^n$ as key mixing function. Some algebraic properties of the round functions can indeed hide  some  weaknesses of the corresponding cipher.
Firstly, in 1975 Coppersmith and Grossman \cite{copp} defined a family of functions which can be used as round functions of a block cipher and studied the permutation group generated by those. Then it has been found out that some group-theoretical properties  can reveal weaknesses of the cipher itself.
For example, if such group is too small, then the cipher is vulnerable  to birthday-paradox attacks (see \cite{kalinski}). Recently, in \cite{Ca15} the authors proved that if such group is of affine type, then it is possible to embed a dangerous trapdoor on the cipher.  More relevant, in \cite{Pat} Paterson built a DES-like cipher whose encryption functions generate an imprimitive group and showed how the knowledge of this trapdoor can be turned into an efficient attack to the cipher.
For this reason, a branch of research in symmetric cryptography is focused on showing that the group generated by the encryption functions of a
given cipher is primitive  (see \cite{PriPre, ACDVS, GOST_ric, cal18, ONAN, CDVS09, SW, We1, We3, We2}). 

Our aim is to  guarantee protection against algebraic attacks based on the imprimitivity of the group generated by the round functions of block ciphers which use addition modulo $2^n$ as key mixing function.  
We do so by identifying a necessary and sufficient property of the structure of the binary matrix associated to the mixing layer, under the only hypothesis of S-Box invertibility.  In particular, we give the definition of {\em type-preserving matrix} and we prove that the group generated by the round functions of an SPN cipher with  addition modulo $2^n$ as key mixing function is primitive if its mixing layer is not type-preserving. 

The paper is organized as follows. In Section \ref{sec:not}, we give our  notation, as well as some basic definitions and results concerning  block ciphers and primitive permutations groups. 
In Section \ref{sec:typepres}  we present a new property for mixing layer, called \emph{non-type-preserving}. Then, after having proved  our result regarding the necessary and sufficient conditions for a mixing layer to be non-type-preserving, we show that some known mixing layers, such as those employed in GOST \cite{GOST}, PRESENT \cite{PRESENT}, AES \cite{AES} and GPig2 \cite{lin-att}, are non-type-preserving. Even though the key mixing of AES and PRESENT is the classical XOR addition instead of the addition modulo $2^n$,  their mixing layers are real-life examples of non-type-preserving matrices. In Section \ref{sec:app},  we prove that an SPN which uses addition modulo $2^n$ as key mixing function and a non-type-preserving matrix as mixing layer is primitive. Finally, we use a non-type-preserving mixing layer to extend a GOST-like cipher, defined in \cite{GOST_ric},  and we prove its primitivity if the S-Boxes are invertible.



\section{Notation and preliminary results}\label{sec:not}
\subsection{Permutation groups}\label{subsec:permgroup}
We recall some basic notions from permutation group theory. Let $G$ be a finite group acting on the set $V$. For each $g \in G$ and $v \in V$ we denote the action of $g$ on $v$ as $vg$. We denote by $vG=\{vg \mid v \in G\}$ the orbit of $v \in V$ and by $G_v=\{g \in G \mid vg=v\}$ its stabilizer.
The group $G$ is said to be \emph{transitive} on $V$ if for each $v,w \in V$ there exists $g \in G$ such that $vg=w$.
A partition $\mathcal{B}$ of $V$ is \emph{trivial} if $\mathcal{B}=\{V\}$ or $\mathcal{B}=\{\{v\} \mid v \in V\}$, and  \emph{$G$-invariant} if for any $B \in \mathcal{B}$ and $g \in G$ it holds $Bg \in \mathcal{B}$. Any non-trivial and $G$-invariant partition $\mathcal{B}$ of $V$ is called a \emph{block system}. In particular any $B \in \mathcal{B}$ is called an \emph{imprimitivity block}. The group $G$ is \emph{primitive} in its action on $V$  if {$G$ is transitive and} there exists no block system. 
Otherwise,  the group $G$ is \emph{imprimitive} in its action on $V$. We remind the following well-known results whose proofs may be found e.g. in \cite{Cam}. 

\begin{lemma}\label{lemma:block}
  A block of imprimitivity is the orbit $v  H$ of a proper subgroup  $H < G$ that properly  contains  the
  stabilizer $G_v$,  for some $v  \in V$.
\end{lemma}

\begin{lemma}\label{lemma:trans}
  If $T$ is a transitive subgroup of $G$, then a block system
  for $G$ is also a block system for $T$.
\end{lemma}

\subsection{Substitution Permutation Networks}
Let $n\in \mathbb N$ and let   $V=\F_2^n$ be  the plaintext space. 
Let  $\sym(V)$ be the symmetric group acting on $V$, i.e. the group of all permutations on $V$, and  by $\AGL(V)$ the group of all affine permutations of $V$, which is a primitive maximal subgroup of $\Sym(V)$, i.e., $\AGL(V)$ is a primitive proper  subgroup such that there is no other primitive proper subgroup containing it.

\noindent A \emph{block cipher} $\mathcal{C}$ is a family of key-dependent permutations of $V$ 
\[\{\veps_K \mid \veps_K: V \rightarrow V, \, K \in \mathcal K \}\subseteq \mathrm{Sym}(V),\]
where $\mathcal K$ is the key space, and $|V|\leq |\mathcal K|$. The permutation $\veps_K$ is called the \emph{encryption function induced by the master key} $K$. Let $\phi:\{1,\ldots,r\}\times\mathcal{K}\longrightarrow V$ be a public procedure known as \emph{key-schedule}, such that $\phi(h,K)$ is the $h$-th round key, given the master key $K$. The block cipher $\mathcal{C}$ is called an {iterated block cipher} if there exists $r \in \mathbb N$ such that for each $K \in \mathcal K$ the encryption function $\veps_K$ is the composition of $r$ {round functions}, i.e. $\veps_K = \veps_{\phi(1,K)}\,\veps_{\phi(2,K)}\ldots\veps_{\phi(r,K)}$. Each round function $\veps_{\phi(h,K)}$ is a permutation of $V$ depending on the $h$-th  \emph{round key}. 

\noindent Most modern iterated block ciphers  belong to two families of cryptosystems: \emph{Substitution Permutation Networks}, briefly \emph{SPN} (see e.g.  SERPENT \cite{SERPENT}, PRESENT \cite{PRESENT}, AES \cite{AES}) and \emph{Feistel Networks}, briefly \emph{FN} (see e.g. Camelia \cite{camelia}, GOST \cite{GOST}). In this paper we mainly deal with ciphers of SPN type and  we  define a class of round functions for iterated block ciphers which is large enough to include the round functions of classical SPN's.

\vspace{2mm}

\noindent Let $V= V_1\times V_2\times \ldots\times V_\delta$ where,  for $1\leq j\leq \delta$, $\dim(V_j) = m$, with $m$ dividing $n$, and $\times$ represents the Cartesian product of vector spaces. The spaces $V_j$'s are called \emph{bricks}.

\begin{definition}
For each $k \in V$, a \emph{round function} induced by $k$ is a map $\veps_{k} \in \sym(V)$ 
where $\veps_{k} = \gamma \lambda \sigma_{k}$ and 
\begin{itemize}
\item $\gamma : V \rightarrow V$ is a non-linear permutation, called \emph{parallel S-Box}, which acts in parallel way on each $V_{j}$, i.e. \[(x_1,x_2,\ldots,x_n)\gamma = \left((x_1,\ldots,x_{m})\gamma_1,\ldots,(x_{m(\delta-1)+1},\ldots,x_{n})\gamma_\delta\right);\] the maps $\gamma_j  :  V_j \rightarrow V_j $ are traditionally called \emph{S-Boxes},
 \item $\lambda \in \sym(V)$ is a linear map, called \emph{mixing layer},
 \item  $\sigma_{k}: V \rightarrow V$ is the  \emph{key mixing function}, that is a permutation of $V$ combining the message with the corresponding round key $k$.
 \end{itemize}
\end{definition}
\noindent Since studying the role of the key-schedule is out of the scopes of this paper, we can simply suppose that round keys are randomly-generated vectors in $V$.

\vspace{2mm} 

\noindent Usually, the key mixing function of well-established SPN's, such as AES, PRESENT, SERPENT, is $\sigma_k: x \mapsto x+k$, where $+$ is the usual bitwise \emph{XOR}. Note that SPN's featuring a XOR-based key addition have been also called \emph{Translation-Based ciphers} in \cite{CDVS09}. In many other ciphers (e.g.  MARS \cite{mars}, GOST \cite{GOST}, IDEA \cite{idea}, RC6 \cite{RC6}, SEA \cite{sea}) the key mixing is the addition modulo $2^m$, for some integer $m$. This kind of key mixing function may be used to increase the nonlinearity of a round function (see for example \cite{lin-att}).  In particular, in this work we are interested in SPN's which combine the message with the key by  the addition modulo $2^{\dim(V)}$ (see \cite{alg-att,lin-att}).  

\begin{definition}
We denote by \emph{SPNmod} an SPN operating on the plaintext space $V$ in which the  key mixing function   is the addition modulo $2^n$, where $n=\dim (V)$.
\end{definition}

\subsection{Group generated by the round functions and Primitivity}
Besides the classical statistical attacks (e.g. differential and linear cryptanalysis), it is proved that also some algebraic attacks can be effective and dangerous (see, for instance, \cite{Ca15,kalinski,Pat}). In this paper we focus on a particular  attack, described in \cite{Pat}, based on the \emph{imprimitivity} of the permutation group generated by the round functions of a block cipher.

\vspace{2mm}

Let $\mathcal{C} = \{\veps_K \mid K \in \mathcal K \} \subseteq \sym(V) $ be an $r$-round iterated block cipher.
Several researchers have shown in recent years that the group generated by the encryption functions of a block cipher
\[
\Gamma(\mathcal{C}) \deq \langle \veps_K \mid K \in \mathcal K\rangle \leq \sym(V)
\]
can reveal weaknesses of the cipher itself (see for example \cite{Ca15,kalinski,Pat}). However, the study of $\Gamma(\mathcal{C})$ is not an easy issue in general, since it strongly depends on the key-schedule function (for an example of a key-schedule related study, see \cite{bea}). 
Hence the research focuses on the group generated by the round functions 
\[
	\Gamma_{\infty}(\mathcal{C})\deq\langle \veps_{h,K}\mid 1 \leq h \leq r, K \in \mathcal{K}\rangle,
\]
where all the possible round keys for round $h$ are considered as varying $K\in\mathcal{K}$. Such group contains $\Gamma(\mathcal{C})$ and  allows to ignore the effect of the key-schedule. 

\vspace{2mm}

In our case $\mathcal{C}$ is an $r$-round SPNmod cipher and the $i$-th round function is  $\veps_{i,K}=\gamma\lambda\sigma_{K_i}$, where $K_i$ is the $i$-th round key derived by the key schedule  
$$
\begin{array}{rcl}
\mathcal{K}&\rightarrow &V^r\\\
K &\mapsto &(K_1,\ldots,K_r),
\end{array}
$$ 
which we suppose  surjective w.r.t. any round.

\noindent The corresponding group generated by the round functions is

\[
	\Gamma_{\infty}\deq\langle \gamma\lambda\sigma_{K_i}\mid 1 \leq i \leq r, K \in \mathcal{K}\rangle.
\]

\noindent 
Throughout this paper, sometimes we will denote $\gamma\lambda$ with $\rho$.\\

   Note that we
can consider two group structures on $V$.  The first operation is the
bitwise XOR, which will be denoted by $\oplus$ and which makes $V$ into a vector space over $\mathbb{F}_2$. 

\noindent The  second  operation,  denoted  by $\boxplus$,  is  the  sum  modulo
$2^{n}$. That is, we represent $a, b \in V$ as
\begin{equation*}
  a = (a_{0}, a_{1}, \dots, a_{n-1}),
  \quad
  b = (b_{0}, b_{1}, \dots, b_{n-1}),
\end{equation*}
with $a_{i}, b_{i} \in \Set{0, 1}$ integers, and let
\begin{equation*}
  a \boxplus b
  =
  (c_{0}, c_{1}, \dots, c_{n-1}),
\end{equation*}
where
\begin{multline*}
  (a_{0} + a_{1} 2 + a_{2} 2^{2} + \dots + a_{n-1} 2^{n-1})
  +
  (b_{0} + b_{1} 2 + b_{2} 2^{2} + \dots + b_{n-1} 2^{n-1})
  \equiv\\\equiv
  c_{0} + c_{1} 2 + c_{2} 2^{2} + \dots + c_{n-1} 2^{n-1}
  \pmod{2^{n}},
\end{multline*}
with $c_{i} \in  \Set{0, 1}$ integers. (Here $+$  denotes the ordinary
sum  of integers.)   Therefore $V$  under $\boxplus$  is equivalent to  the group  $\Z_{2^{n}}$ of integers  modulo $2^{n}$,  and we
will denote  it by $(\Z_{2^{n}},  \boxplus)$. 

\noindent We recall the following elementary fact we will be using repeatedly
without further mention.

\begin{lemma}\label{lemma:subgroups}
  
  The  subgroups  of   $(\Z_{2^{n}},  \boxplus)$  are  linearly
    ordered; they are  $\Span{2^{q}}$, for $0 \le q \le n$.
\end{lemma}

Now we prove the first property of the group generated by the round functions of an SPNmod cipher. Let
\begin{equation*}
 \mathcal{T}\;\deq\; \mathcal{T}(V)=
  \Set{
     \sigma_k: v\mapsto v\boxplus k \mid k\in V
   }
\end{equation*}
be the group of $\boxplus$-translations on $V$. Note that $\mathcal{T}$ transitively acts on $V$.
\begin{lemma}\label{lem:spanTrho}
\begin{equation*}
  \Gamma_{\infty}
  =
  \Span{
    \mathcal{T}, \rho
    }.
\end{equation*}
In particular,
$\Gamma_{\infty}$ acts transitively on $V$.
\end{lemma}
\begin{proof}
If we set $k=0$, then $\rho\sigma_0=\rho\in\Gamma_{\infty}$, and so $\rho^{-1}\in\Gamma_{\infty}$. Finally for all $k\in V$, we have $\rho^{-1}\rho\sigma_k=\sigma_k\in\Gamma_{\infty}$.
\end{proof}


%
%

\noindent Since the map $v \mapsto \sigma_{v}$ is an isomorphism $(V,
\boxplus) \to \mathcal{T}$, so we have the following well known result
\begin{lemma}[\cite{Cam}]\label{lemma:trivobs}
  The subgroups of $\mathcal{T}$ are of the form 
  \begin{equation*}\label{eq:a-subgroup}
    \Set{ \sigma_{u} : u \in U },
  \end{equation*}
  where $U$ is a subgroup of $(V, \boxplus)$.
\end{lemma}
\begin{lemma}[\cite{Cam}]\label{rem1}
  If $\Gamma_{\infty}$ acting on $V$ has  a block system, then this consists of
  the cosets  of a $\boxplus$-subgroup of  $V$, that is, it  is of the
  form
  \begin{equation*}
    \Set{
      W \boxplus v
      :
      v \in V
      },
  \end{equation*}
  where $W$ is a non-trivial, proper subgroup of $(V, \boxplus)$.
\end{lemma}

 
\subsubsection*{Imprimitivity attack}
  The cryptanalysts' interest into the imprimitivity of the group generated by the
round functions of a block cipher arises from the study performed in \cite{Pat}, where it is
shown how the imprimitivity of the group can be exploited to construct a trapdoor
that may be hard to detect. In particular, the author gives an example of a DES-like
cipher which can be easily broken since its round functions generate an imprimitive
group, but which is resistant to both linear and differential cryptanalysis.
  


\subsection{Some other  definitions and known results}


Now we
will recall some preliminary results proved in \cite{GOST_ric}, and to do so we will adopt the same notation introduced
therein.



\noindent We shall denote
\begin{itemize}
\item a subset of $\F_{2}^{m}$ of cardinality 1 by a \emph{white box};
%

\item a subset of $\F_{2}^{m}$ of cardinality $1<t<2^{m}$ by a \emph{ruled box};
%

\item the full set $\F_{2}^{m}$ by a \emph{black box}.
%
\end{itemize}
We will say that a box has white, ruled or black \emph{type}.

%
%

\begin{definition}\label{def:fortype}
  Let $D$ be a subset of
  \begin{equation*}
    \F_{2}^{n} 
    =
    V_{1} \times V_{2} \times \dots \times V_{\delta},
  \end{equation*}
  where each space $V_{i}$ has dimension $m$.
  The \emph{type}  of $D$ will  be a  sequence of
  $\delta$  white,  ruled  or  black   boxes,  where  the  $i$-th  box
  represents the projection of $D$ on $V_i$.
\end{definition}

\begin{remark}[Remark 4.9 in \cite{GOST_ric}]\label{rem:typesofsubgroups}
 According to Lemma~\ref{lemma:subgroups}, a subgroup  $D$ of
$\mathbb{Z}_{2^{n}}$  is of the form $\Span{2^{q}}$, for some $0 \le q <
  n$. Hence a subgroup $D = \Span{2^{q}}$ of
  $\mathbb{Z}_{2^n}$ has one of the following two types.

  \begin{enumerate}
  \item When $q \equiv 0 \pmod{m}$, the subgroup has $n_w$ white boxes and $\delta-n_w$ black boxes, where $0\leq n_w\leq \delta$ such that $q=n_w m$. Note that there are no white boxes when $q = 0$  (the subgroup is the full group  $\Z_{2^{n}}$), and there are no black boxes when $q = n$ (the subgroup is $\Set{0}$).

\item When $q \not\equiv 0 \pmod{m}$, there is 
  a ruled box which is the box containing the $q$-th bit.
  \end{enumerate}
\end{remark}

\noindent Due to Remark \ref{rem:typesofsubgroups}, we can associate to the type of any subgroup $D$ in $\mathbb{Z}_{2^n}$ the triple $(n_w,n_r,n_b)$, where $n_w,\;n_r$ and $n_b$ are respectively the number of white, ruled and black boxes. We have the following bounds:
\begin{equation}\label{eq: bounds type W}
\left\{
\begin{array}{l}
n_w+n_r+n_b=\delta\\
0\leq n_w\leq \delta\\
0\leq n_r\leq 1\\
0\leq n_b\leq \delta
\end{array}
\right.
\end{equation}
With a slight abuse of notation, we use the triple $(n_w,n_r,n_b)$ to denote the type of $D$.



In the next lemma, proved in \cite{GOST_ric}, we  consider  the behavior of the modular sum $\boxplus$ with
respect to types. 

\begin{lemma}\label{lemma:Wform}
  If  $D$  is  a  subgroup  of $\mathbb{Z}_{2^n}$ and
  $v\in\mathbb{Z}_{2^n}$,  then $D$  and 
  $v\boxplus D$ have the same type.
\end{lemma}


\section{Type-preserving matrices}\label{sec:typepres}

In this section we study the diffusion properties of an invertible mixing layer $\lambda$, namely how the multiplication by a full-rank binary matrix $\Lambda$ mixes the bricks
$V_1,\ldots,V_\delta$. To do so, we consider $\Lambda$ to be a $\delta\times \delta$ block matrix whose blocks are binary square matrices of order $m$:
$$
\Lambda=\begin{bmatrix}
\Lambda_{1,1} &\cdots & \Lambda_{1,\delta}\\
\vdots & & \vdots\\
\Lambda_{\delta,1} & \cdots & \Lambda_{\delta,\delta}
\end{bmatrix}.
$$
We will also use the notation $\Lambda_{(i_1,j_1):(i_2,j_2)}$ for the submatrices of $\Lambda$:
$$
\Lambda_{(i_1,j_1):(i_2,j_2)}:=\
\begin{bmatrix}
\Lambda_{i_1,j_1}&\cdots&\Lambda_{i_1,j_2}\\
\vdots &&\vdots\\
\Lambda_{i_2,j_1}&\cdots&\Lambda_{i_2,j_2}
\end{bmatrix}.
$$
Observe that if $\Lambda_{i,j}=0$ whenever $i\neq j$, i.e. $\Lambda$ is a diagonal block matrix, then $\gamma\lambda$ is a parallel map. 
\\
Our interest lies in the image of $D\subseteq\F_2^n$ through the mixing layer, thus we will work with the set $\mathrm{Im}_{|_D}\lambda=\{v\Lambda\;: v\in D\}$. In many cases we will need to work with submatrices of $\Lambda$, and for the sake of simplicity we will write $\mathrm{Im}_{|_D}\Lambda_{(i_1,j_1):(i_2,j_2)}$ to denote the restriction of the image $\mathrm{Im}\left(\Lambda_{(i_1,j_1):(i_2,j_2)}\right)$ to the set obtained by projecting $D$ on the coordinates corresponding to the boxes $j_1,\ldots,j_2$.\\
We will study which properties of $\Lambda$ imply
\begin{equation}\label{eq: claim}
\mathtt{type}\left(\mathrm{Im}_{|_D}\lambda\right)= \mathtt{type}(D) .
\end{equation}
\begin{definition}\label{def:typepres}
A matrix $\Lambda\in\mathrm{GL}(\F_2^n)$, or equivalently the corresponding mixing layer $\lambda$, satisfying equation \eqref{eq: claim} for any $D\subseteq\F_2^n$, is called \emph{type-preserving}. Vice versa, if $\Lambda$ is not type-preserving, then we say that it is  \emph{non-type-preserving}.
\end{definition}
\begin{remark}
In Section \ref{sec:app} we prove that the non-type-preserving property of a mixing layer given in the previous definition is useful to avoid imprimitivity attacks on block ciphers with the following structure:
\begin{itemize}
\item SPN with addition $2^n$ as key mixing function (Theorem \ref{secSPNmod}),
\item GOST-like with addition $2^{n/2}$ as key mixing function and invertible S-Boxes (Theorem \ref{secGOSTlike}),
\end{itemize}
where $n$ is the length of the whole block.
\end{remark}


In this paper we are mainly  interested in the subsets $D$ of $\F_{2}^{n}$, such as the subgroups of $\Z_{2^n}$, with  type $(n_w,n_r,n_b)$ satisfying  equation \eqref{eq: bounds type W}. Therefore in the remaining part of this section the subsets $D$ of $\F_2^n$ are all of this kind. Observe that any $v\in D$ can be written as the concatenation $(v_w|v_r|v_b)$, where the lengths of $v_w$, $v_r$ and $v_b$ are determined by the type of $D$. In particular, $v_w\in\mathbb{F}_2^{mn_w}$, $v_r\in\mathbb{F}_2^{mn_r}$ and $v_b\in\mathbb{F}_2^{mn_b}$, with the following properties due to the structure of $D$:
\begin{equation}\nonumber
\left\{
\begin{array}{l}
\left|\{v_w\;: \exists v=(v_w|v_r|v_b)\in D\}\right|=1\\
2\leq \left|\{v_r\;: \exists v=(v_w|v_r|v_b)\in D\}\right|\leq 2^{mn_r}-1\\
\left|\{v_b\;: \exists v=(v_w|v_r|v_b)\in D\}\right|=2^{mn_b}.
 \end{array}
 \right.
 \end{equation}
 
 
 Now we can state our main result, whose proof is a consequence of several lemmas. 
 \begin{theorem}\label{thm: properties of M}
The mixing layer $\lambda$ is type-preserving with respect to the subsets of $\F_{2}^{n}$ with  type $(n_w,n_r,n_b)$ satisfying  equation \eqref{eq: bounds type W} if and only if there exists an integer $n_w\in\{0,\ldots,\delta\}$ for which either equation 
\begin{equation}\label{eq: property for nonruled}
\Lambda_{(n_w+1,1):(\delta,n_w)}=0
\end{equation} 
or the following four properties 
\begin{enumerate}[(a)]
\item\label{first property} $\Lambda_{(n_w+2,1):(\delta,n_w)}=0$,
\item\label{second property} $\Lambda_{(n_w+1,1):(n_w+1,n_w)}$ is not a full-rank matrix,
\item\label{third property} $2\leq \left|\mathrm{Im}_{|_{D}}\left(\Lambda_{(n_w+1,n_w+1):(\delta,n_w+1)}\right)\right|\; < \;2^{m\left(\delta-n_w-1\right)},$
\item\label{fourth property} $\left|\mathrm{Im}_{|_{D}}\left(\Lambda_{(n_w+1,n_w+2):(\delta,\delta)}\right)\right|\; =\;2^{m\left(\delta-n_w-1\right)}$
\end{enumerate}  
are satisfied.
\end{theorem}
 
\begin{proof}
\noindent By equation \eqref{eq: bounds type W} we have four cases:
\begin{enumerate}
\item\label{case 1} $\mathtt{type}(D)=(0,0,\delta)$
\item\label{case 2} $\mathtt{type}(D)=(\delta,0,0)$
\item\label{case 3} $\mathtt{type}(D)=(n_w,0,\delta-n_w)$
\item\label{case 4} $\mathtt{type}(D)=(n_w,1,\delta-n_w-1)$
\end{enumerate}
Cases \ref{case 1} and \ref{case 2} are trivial, namely all invertible linear maps, i.e. all full-rank matrices, preserve these types: $\mathtt{type}(D)=(\delta,0,0)$ implies $\mathtt{type}\left(\mathrm{Im}_{|_D}\lambda\right)=(\delta,0,0)$, and $\mathtt{type}\left(D\right)=(0,0,b)$ implies $\mathtt{type}\left(\mathrm{Im}_{|_D}\lambda\right)=(0,0,\delta)$. 
We will focus on the remaining two cases, starting by case \ref{case 3}.
\begin{lemma}\label{lem: necessary nonruled}
Let $\mathtt{type}(D)=\mathtt{type}\left(\mathrm{Im}_{|_D}\lambda\right)=(n_w,0,\delta-n_w)$, where $n_w\in\{1,\ldots,\delta-1\}$. Then 
\begin{equation}
\Lambda_{(n_w+1,1):(\delta,n_w)}=0.
\end{equation}
\end{lemma}
\begin{proof}[Proof of Lemma \ref{lem: necessary nonruled}]
We assume $\mathtt{type}(D)=(n_w,0,\delta-n_w)$ and $\Lambda_{(n_w+1,1):(\delta,n_w)}\neq0$. We consider two vectors $v=(v_w|v_b)$ and $v'=(v_w'|v_b')$ in $D$, with $v\in\mathrm{ker}\left(\Lambda_{(n_w+1,1):(\delta,n_w)}\right)$ while $v'$ is outside of it. Observe that the structure of $D$ implies that $v_w=v_w'$, and by applying $\lambda$ to both we obtain $v\Lambda=(v_w\Lambda_{(1,1):(n_w,n_w)}|0)$ and $v'\Lambda=v\Lambda\oplus(0|v_b\Lambda_{(n_w+1,1):(\delta,n_w)})$, here in both cases 0 denotes a string of $n_b$ zeros. Since the two vectors are different, $\mathtt{type}\left(\mathrm{Im}_{|_D}\lambda\right)\neq (n_w,0,\delta-n_w)$, which contradicts the hypotheses of the Lemma.
\end{proof}
The above lemma gives us a necessary property on $\Lambda$ to have a mixing layer which preserves the type $(n_w,0,\delta-n_w)$. The next result assures  that this is also sufficient.
\begin{lemma}\label{lem: sufficient nonruled}
Let $n_w\in\{1,\ldots,\delta-1\}$ and $\Lambda_{(n_w+1,1):(\delta,n_w)}=0$. Then $\lambda$ preserves the type $(n_w,0,\delta-n_w)$.
\end{lemma}
\begin{proof}[Proof of Lemma \ref{lem: sufficient nonruled}]
We construct $D$ so that its type would be $(n_w,0,\delta-n_w)$. Then, any vector $v\in D$ can be written as a concatenation $(v_w|v_b)$, where $v_w$ is fixed, while 
\begin{equation}\label{eq: dimension black}
\{v_b \;:\;\exists v=(v_w|v_b)\in D\}=\mathbb{F}_2^{\delta-n_w} .
\end{equation} 
Due to $\Lambda_{(n_w+1,1):(\delta,n_w)}$ being the zero matrix, the first $mn_w$ bits of the image of any $v\in D$ are equal to $v_w\Lambda_{(1,1):(n_w,n_w)}$, hence the first $n_w$ boxes of $\mathrm{Im}_{|_D}\lambda$ are white. On the other hand, since $\lambda$ is invertible, $\Lambda$ has full rank, which can only be possible  by assuming that  $\Lambda_{(n_w+1,n_w+1):(\delta,\delta)}$ is invertible. By equation \eqref{eq: dimension black}, we therefore have $\{v_b\Lambda_{(n_w+1,n_w+1):(\delta,\delta)}\;:\; \exists v=(v_w|v_b)\in D\}=\mathbb{F}_2^{mn_w}$, from which we conclude that $\mathtt{type}\left(\mathrm{Im}_{|_D}\lambda\right)=(n_w,0,\delta-n_w)$.
\end{proof}

Note that in Lemma \ref{lem: necessary nonruled} and Lemma \ref{lem: sufficient nonruled} we did not consider the cases $n_w=0$ and $n_w=\delta$, because they respectively correspond to the cases \ref{case 1} and \ref{case 2} which we have already discussed.

At last,  the case \ref{case 4}, $\mathtt{type}(D)=(n_w,1,\delta-n_w-1)$.
\begin{lemma}\label{lem: necessary ruled}
Let both $D$ and $\mathrm{Im}_{|_D}\lambda$ be of type $(n_w,1,\delta-n_w-1)$, where $n_w\in\{1,\ldots,\delta-2\}$.
Then $\Lambda$ satisfies the following properties:
\begin{enumerate}[(a)]
\item
$\Lambda_{(n_w+2,1):(\delta,n_w)}=0$,
\item
$\Lambda_{(n_w+1,1):(n_w+1,n_w)}$ is not a full-rank matrix,
\item
$2\leq \left|\mathrm{Im}_{|_{D}}\left(\Lambda_{(n_w+1,n_w+1):(\delta,n_w+1)}\right)\right|\; < \;2^{m\left(\delta-n_w-1\right)}$,
\item
$\left|\mathrm{Im}_{|_{D}}\left(\Lambda_{(n_w+1,n_w+2):(\delta,\delta)}\right)\right|\; =\;2^{m\left(\delta-n_w-1\right)}$.
\end{enumerate} 
\end{lemma}
\begin{proof}[Proof of Lemma \ref{lem: necessary ruled}]
We proceed in four steps, assuming each time that a property among \eqref{first property}, \eqref{second property}, \eqref{third property} and \eqref{fourth property} would not be necessary.
We use again the notation $v=(v_w|v_r|v_b)$, where the length of the three vectors depends on the type of $D$, and we recall that $v_w$ is the same for each $v\in D$.
\\
Firstly, we look at what happens if we deny property \eqref{first property}.  In this case, we consider $v=(v_w|v_r|v_b)$ and $v'=(v_w|v_r|v_b')$ in $D$ with $v_b\in\mathrm{ker}\Lambda_{(n_w+2,1):(\delta,n_w)}$ and $v_b'\notin\mathrm{ker}\Lambda_{(n_w+2,1):(\delta,n_w)}$. It follows that the first $mn_w$ bits of $v_b\Lambda$ are different from the first $mn_w$ bits in $v_b\Lambda$, hence  the first $n_w$ boxes in $\mathrm{Im}_{|_D}\lambda$ are not white, and so the type of $D$ is not $(n_w,1,\delta-n_w-1)$.
\\
Similarly, if we deny the second property, we have the same conclusion by choosing $v=(v_w|v_r|v_b)$ and $v'=(v_w|v_r'|v_b)$, with $v_r\neq v_r'$.
\\
We do not go through the entire proofs  of Properties \eqref{third property} and \eqref{fourth property}, since they are quite similar to what we already did above. The difference is that we need to use the entire $D$ instead of just two vectors $v$ and $v'$, and therefore prove that $\mathrm{Im}_{|_{D}}\lambda$ does not have respectively a ruled box (by denying property \eqref{third property}) and the right number of black boxes (by denying property \eqref{fourth property}).
\end{proof}
As we did for Lemma \ref{lem: necessary nonruled}, we can also prove that the four necessary properties in Lemma \ref{lem: necessary ruled} are also sufficient.
\begin{lemma}\label{lem: sufficient ruled}
Let $\Lambda$ be a matrix satisfying the four properties in Lemma \ref{lem: necessary ruled} for a certain integer $n_w\in\{1,\ldots,\delta-2\}$. Then $\lambda$ preserve the type $(n_w,1,\delta-n_w-1)$.
\end{lemma}
\begin{proof}[Proof of Lemma \ref{lem: sufficient ruled}]
We consider $D$ of type $(n_w,1,\delta-n_w-1)$, where its ruled box is the kernel of the matrix $\Lambda_{(n_w+1,1):(n_w+1,n_w)}$.
\end{proof}


 Observe that in Lemma \ref{lem: necessary ruled} and Lemma \ref{lem: sufficient ruled} we did not consider $n_w=0$ and $n_w=\delta-1$. We discuss these cases in the following two results.

\begin{lemma}\label{lem: 01delta}
Let $\mathtt{type}(D)=\mathtt{type}\left(\mathrm{Im}_{|_D}\lambda\right)=(0,1,\delta-1)$. Then $\Lambda$ satisfies Properties \eqref{third property} and \eqref{fourth property} of Lemma \ref{lem: necessary nonruled}, namely
\begin{enumerate}[(a)]
\setcounter{enumi}{2}
\item\label{third property 01delta} $2\leq \left|\mathrm{Im}_{|_{D}}\left(\Lambda_{(1,1):(\delta,1)}\right)\right|\; < \;2^{m\left(\delta-1\right)}$, and
\item\label{fourth property 01delta} $\left|\mathrm{Im}_{|_{D}}\left(\Lambda_{(1,2):(\delta,\delta)}\right)\right|\; =\;2^{m\left(\delta-1\right)}$.
\end{enumerate} 
Conversely, if there exist $D$ of type $(0,1,\delta-1)$ for which $\Lambda$ satisfies the two properties above, then $\Lambda$ preserves the type of $D$.
\end{lemma}
\begin{lemma}\label{lem: delta10}
Let $\mathtt{type}(D)=\mathtt{type}\left(\mathrm{Im}_{|_D}\lambda\right)=(\delta-1,1,0)$. Then $\Lambda$ satisfies Properties \eqref{second property} and \eqref{third property} of Lemma \ref{lem: necessary nonruled}, namely
\begin{enumerate}[(a)]
\setcounter{enumi}{1}
\item\label{second property delta10} $\Lambda_{(\delta,1):(\delta,\delta-1)}$ is not a full-rank matrix, and
\item\label{third property delta10} $\Lambda_{\delta,\delta}\neq 0$.
\end{enumerate} 
Conversely, if $\Lambda$ satisfies the two properties above, then there exists $D$ whose type is preserved by $\Lambda$.
\end{lemma}
Note that the properties described in Lemmas \ref{lem: 01delta} and \ref{lem: delta10} are particular cases of the ones presented in Lemma \ref{lem: necessary ruled}. We omit the proofs of these lemmas, since they can be obtained using the same arguments applied to prove Lemma \ref{lem: necessary ruled} and Lemma \ref{lem: sufficient ruled}. Hence, we denoted the new properties in the same way, and, with a slight abuse of notation, in the following we will simply refer to Lemma \ref{lem: necessary ruled} and its properties, even though when speaking of types $(0,1,\delta-1)$ and $(\delta-1,1,0)$ we should be careful and use the dedicated results.

Putting everything together, we obtain the proof of Theorem \ref{thm: properties of M} as a straightforward consequence of Lemmas \ref{lem: necessary nonruled}, \ref{lem: sufficient nonruled}, \ref{lem: necessary ruled}, \ref{lem: sufficient ruled}, \ref{lem: 01delta} and \ref{lem: delta10}. 
\end{proof}


We remark that many matrices often used to obtain mixing layers are non-type-preserving, simply because they usually do not satisfy property \eqref{first property} of Lemma \ref{lem: necessary ruled}.
\begin{corollary}\label{lem:anontypepres}
If $\Lambda_{(n_w+2,1):(\delta,n_w)}\ne 0$, for any $n_w\in\{1,\ldots,\delta-2\}$, then $\Lambda$ is non-type-preserving.
\end{corollary}
\begin{proof}
By Theorem \ref{thm: properties of M}, if both equation \eqref{eq: property for nonruled} and property \eqref{first property} are not satisfied, then $\Lambda$ is non-type-preserving. Note that $\Lambda_{(n_w+2,1):(\delta,n_w)}$ is a submatrix of $\Lambda_{(n_w+1,1):(\delta,n_w)}$, so $\Lambda_{(n_w+2,1):(\delta,n_w)}\ne 0$ implies $\Lambda_{(n_w+1,1):(\delta,n_w)}\ne 0$.
\end{proof}
In the next section we show how some known families of mixing layers are non-type-preserving with respect to the subsets of $\F_{2}^{n}$ with  type $(n_w,n_r,n_b)$ satisfying  equation \eqref{eq: bounds type W}.

\subsection{Examples of non-type-preserving mixing layers}
In this section we characterize some known classes of mixing layers by proving whether they are non-type-preserving with respect to the subsets of $\F_2^n$ whose type satisfy equation \eqref{eq: bounds type W}. 
The aim of this section is to highlight that the definition of non-type-preserving mixing layer is not restrictive. Indeed, in many real-life ciphers, such as GOST, PRESENT and AES, such kind of mixing layers are used. With a slight abuse of notation, any of these mixing layers will simply be denoted as non-type-preserving.

\subsubsection*{Rotation of a GOST-like cipher}
In   \cite{GOST_ric}, the mixing layer of a $\mathrm{GOST}$-like cipher is defined as the permutation matrix $\Lambda_s$ with $s\in\{m,\ldots,(\delta-1)m\}$.

\noindent Let $\{\mathbf{e}_1,\ldots,\mathbf{e}_n\}$ be the canonical basis of $\F_2^n$. 
\begin{definition}
Let $\pi_s\in\mathrm{Sym}(n)$ be the permutation defined by
$$
\pi_s=\left(\begin{array}{cccc}
1 & 2 & \ldots & n \\
\pi_s(1) & \pi_s(2) & \ldots & \pi_s(n)
\end{array}\right) 
$$
such that, for each $1\leq x \leq n$,
\begin{equation}\label{def:pi}
\pi_s(x)=x+s \mod n
\end{equation}
where $m\leq s \leq (\delta-1)m$.

\noindent The permutation binary matrix associated to $\pi$ is the following circulant matrix
$$
\Lambda_s=\begin{bmatrix}
\mathbf{e}_{\pi_s(1)} \\ \vdots \\ \mathbf{e}_{\pi_s(n)} 
\end{bmatrix}.
$$

\end{definition}


\begin{example}
In the case of $\mathrm{GOST}$, the actual values of the parameters are: $n =
32$, $m = 4$, $\delta= 8$ and $s = 11$. The right rotation by 11 bits of the $\mathrm{GOST}$ cipher is the permutation matrix associated to the following permutation of $32$ bits:
$$
\pi_{11}=\left(\begin{array}{cccc}
1 & 2 & \ldots & 32 \\
12 & 13 & \ldots & 11
\end{array}\right).
$$

\noindent The mixing layer associated to $\pi_{11}$ is
$$
\Lambda_{\mathrm{GOST}}=\Lambda_{11}=\begin{bmatrix}
0 & \mathbb{1}_{21} \\
\mathbb{1}_{11} & 0
\end{bmatrix},
$$
where we denote the $r\times t$ zero matrix by $0$ and the $t\times t$ identity matrix by $\mathbb{1}_{t}$.
\end{example}

\begin{proposition}
Let $\Lambda_s$ be a binary circulant permutation matrix associated to the rotation of $s$ bits. Then $\Lambda_s$ is non-type-preserving if and only if $m\leq s\leq m(\delta-1)$.
\end{proposition}
\begin{proof}
We write $\Lambda$ as the block matrix
$$
\Lambda=\begin{bmatrix}
0 & \mathbb{1}_{m\delta-s}\\
\mathbb{1}_{s} & 0
\end{bmatrix} ,
$$
where $\mathbb{1}_t$ is the $t\times t$ identity matrix. We will deal with several cases independently, starting by $s=m$. 
\\
In this case, for each $n_w\in\{1,\ldots,\delta-1\}$ we have $\Lambda_{n_w+1,n_w}=\mathbb{1}_m$, hence equation \eqref{eq: property for nonruled} is never satisfied. Moreover, it follows that also property \eqref{second property} is never satisfied. So, the only possibility left is that $\Lambda$ satisfies both property \eqref{third property 01delta} and property \eqref{fourth property 01delta} of Lemma \ref{lem: 01delta}, so that $\Lambda$ would preserve a certain set $D$ of type $(0,1,\delta-1)$. However, since $\Lambda_{1,2}=\mathbb{1}_m$, it follows that property \eqref{third property 01delta} cannot be satisfied by a set of such type.
\\
Let now $s$ be strictly larger than $m$. Then, property \eqref{second property} is never satisfied, hence we only need to deal with Lemma \ref{lem: 01delta}. Note that we can still apply the same argument as we did above, and therefore prove that property \eqref{third property 01delta} cannot be applied.
\\
These two cases together prove that for any $s\in\{m,\ldots,m(\delta-1)\}$ the rotation of $s$ bits is non-type-preserving. We assume now that $s$ is not inside the interval, and prove that $\Lambda$ is a type-preserving matrix. Trivially, if $s=0$ then $\Lambda$ is the identity matrix, which is a type-preserving matrix. In the other possible cases,  $\Lambda_{(\delta,1):(\delta,\delta-1)}$ is not a full-rank matrix, and $\Lambda_{\delta,\delta}\neq 0$. Then, $\Lambda$ satisfies respectively property \eqref{second property delta10} and property \eqref{third property delta10} of Lemma \ref{lem: delta10}, implying that $\Lambda$ is type-preserving.
\end{proof}

\begin{corollary}
The mixing layer of a $\mathrm{GOST}-\mathrm{like}$ cipher is non-type-preserving.
\end{corollary}

\subsubsection*{Mixing layer of PRESENT}
The mixing layer  $\Lambda_{\mathrm{PRESENT}}$ of PRESENT  (see \cite{PRESENT}) is a permutation matrix in $\mathrm{GL_{64}(\F_2)}$ defined by 
$$
\pi(i)=\left\{\begin{array}{ll}
(16(i-1) \mod 63) +1 & \mbox{ if } 1\leq i \leq 63\\
64 & \mbox{ if } i=64.
\end{array}\right.
$$
\begin{lemma}
The mixing layer of $\mathrm{PRESENT}$ is non-type-preserving.
\end{lemma}
\begin{proof}
First, recall that in PRESENT we have 16 bricks of dimension 4. Note that 
\begin{itemize}
\item the bit of value $1$ in position $(13,4)$ is contained in the submatrices \linebreak $(\Lambda_{\mathrm{PRESENT}})_{(3,1):(16,1)}$ and $(\Lambda_{\mathrm{PRESENT}})_{(4,1):(16,2)}$;
\item the bit of value $1$ in position $(45,12)$ is contained in the submatrices \linebreak $(\Lambda_{\mathrm{PRESENT}})_{(n_w+2,1):(16,n_w)}$, for each $n_w\in\{3,\ldots,10\}$;
\item the bit of value $1$ in position $(61,16)$ is contained in the submatrices \linebreak $(\Lambda_{\mathrm{PRESENT}})_{(n_w+2,1):(16,n_w)}$, for each $n_w\in\{11,\ldots,14\}$.
\end{itemize}
So $(\Lambda_{\mathrm{PRESENT}})_{(n_w+2,1):(16,n_w)}\ne 0$, for each $n_w\in\{1,\ldots,14\}$ and hence we can apply Corollary \ref{lem:anontypepres}.
\end{proof}

\subsubsection*{MDS matrix}
\begin{definition}
A matrix over a finite field which has all the minors not equal to zero is called \emph{MDS (Maximum Distance Separable)}.
\end{definition}
\begin{lemma}
An MDS mixing layer $\Lambda$ over $\F_{2^m}$, with $m>1$ the dimension of each S-Box, is non-type-preserving.
\end{lemma}
\begin{proof}
By definition it follows $\Lambda_{(n_{w}+2,1):(\delta,n_{w})} \neq 0$ for each $n_w \in \{1,\ldots,\delta-2\}$, hence we can apply Corollary \ref{lem:anontypepres}.
\end{proof}
\subsubsection*{Mixing layer of an AES-like cipher}
Let 
$$
\texttt{MixColumns}=
\begin{bmatrix}
M &  \cdots & 0\\
\vdots & \ddots & \vdots\\
0 & \cdots & M
\end{bmatrix}\in\mathrm{GL}_{\delta}(\F_{2^m}),
$$
where $\delta=2^t$, for some even integer $t$, and we write the matrix as a $2^{t/2} \times 2^{t/2}$ block matrix with each block being in $\mathrm{GL}_{2^{t/2}}(\F_{2^m})$; in particular, 0 is the zero matrix in $\mathrm{GL}_{2^{t/2}}(\F_{2^m})$ and $M$ is an MDS matrix in $\mathrm{GL}_{2^{t/2}}(\F_{2^m})$.

\noindent  With the same notation as above, let
$$
\texttt{ShiftRows}=
\begin{bmatrix}
\mathbb{I}_{1}& \mathbb{I}_{2} & \ldots &\mathbb{I}_{2^{t/2}-1} &\mathbb{I}_{2^{t/2}}\\
\mathbb{I}_{2^{t/2}}&\mathbb{I}_{2}&\ldots&\mathbb{I}_{2^{t/2}-2}&\mathbb{I}_{2^{t/2}-1}\\
\vdots&\vdots&\ddots&\vdots&\vdots\\
\mathbb{I}_{2}&\mathbb{I}_{3}&\cdots&\mathbb{I}_{2^{t/2}}&\mathbb{I}_{1}
\end{bmatrix}\in\mathrm{GL}_{\delta}(\F_{2^m})
$$
be a circulant block matrix, where $\mathbb{I}_{j}$ is the matrix in $\mathrm{GL}_{2^{t/2}}(\F_{2^m})$ with the identity element of $\F_{2^m}$ in position $(j,j)$ and the zero element of $\F_{2^m}$ everywhere else.

\noindent Let us define $\Lambda$ as the following block matrix in $\mathrm{GL}_{\delta}(\F_{2^m})$
$$
\Lambda=\texttt{ShiftRows}^{\mathrm{T}}\cdot\texttt{MixColumns}^{\mathrm{T}}=
\begin{bmatrix}
\mathbb{I}_{1}\cdot M^{\mathrm{T}}& \mathbb{I}_{2^{t/2}}\cdot M^{\mathrm{T}} & \ldots &\mathbb{I}_{3}\cdot M^{\mathrm{T}} &\mathbb{I}_{2}\cdot M^{\mathrm{T}}\\
\mathbb{I}_{2}\cdot M^{\mathrm{T}}&\mathbb{I}_{1}\cdot M^{\mathrm{T}}&\ldots&\mathbb{I}_{4}\cdot M^{\mathrm{T}}&\mathbb{I}_{3}\cdot M^{\mathrm{T}}\\
\vdots&\vdots&\ddots&\vdots&\vdots\\
\mathbb{I}_{2^{t/2}}\cdot M^{\mathrm{T}}&\mathbb{I}_{2^{t/2}-1}\cdot M^{\mathrm{T}}&\cdots&\mathbb{I}_{2}\cdot M^{\mathrm{T}}&\mathbb{I}_{1}\cdot M^{\mathrm{T}}
\end{bmatrix}.
$$
\begin{example}
In the case of $\mathrm{AES}$ we have 16 bricks of dimension $8$, that is, $\delta=16$ and $m=8$. 
Let 
$$
\emph{\texttt{MixColumns}}_{\mathrm{AES}}=
\begin{bmatrix}
M & 0 & 0 & 0\\
0 & M & 0 & 0\\
0 & 0 & M & 0\\
0 & 0 & 0 & M
\end{bmatrix}\in\mathrm{GL}_{16}(\F_{2^8}),
$$
where we write it as a $4\times4$ block matrix with each block being in $\mathrm{GL}_{4}(\F_{2^8})$; in particular, 0 is the zero matrix in $\mathrm{GL}_{4}(\F_{2^8})$ and
$$
M=
\begin{bmatrix}
\hex{2} & \hex{3} & \hex{1} & \hex{1}\\
\hex{1} & \hex{2} & \hex{3} & \hex{1}\\
\hex{1} & \hex{1} & \hex{2} & \hex{3}\\
\hex{3} & \hex{1} & \hex{1} & \hex{2}
\end{bmatrix}\in\mathrm{GL}_{4}(\F_{2^8}) 
$$
using the hexadecimal notation.

\noindent  With the same notation above, let
$$
\emph{\texttt{ShiftRows}}_{\mathrm{AES}}=
\begin{bmatrix}
\mathbb{I}_{1}& \mathbb{I}_{2} &\mathbb{I}_{3} &\mathbb{I}_{4}\\
\mathbb{I}_{4}&\mathbb{I}_{1}&\mathbb{I}_{2}&\mathbb{I}_{3}\\
\mathbb{I}_{3}&\mathbb{I}_{4}&\mathbb{I}_{1}&\mathbb{I}_{2}\\
\mathbb{I}_{2}&\mathbb{I}_{3}&\mathbb{I}_{4}&\mathbb{I}_{1}
\end{bmatrix}\in\mathrm{GL}_{16}(\F_{2^8}),
$$
where $\mathbb{I}_{j}$ is the matrix in $\mathrm{GL}_{4}(\F_{2^8})$ with $\hex{1}$ in position $(j,j)$ and $\hex{0}$ everywhere else.

\noindent The mixing layer of $\mathrm{AES}$ is the following matrix in $\mathrm{GL}_{16}(\F_{2^8})$
$$
\Lambda_{\mathrm{AES}}=\emph{\texttt{MixColumns}}_{\mathrm{AES}}\cdot\emph{\texttt{ShiftRows}}_{\mathrm{AES}}=
\begin{bmatrix}
M\cdot\mathbb{I}_{1}& M\cdot\mathbb{I}_{2} &M\cdot\mathbb{I}_{3} &M\cdot\mathbb{I}_{4}\\
M\cdot\mathbb{I}_{4}&M\cdot\mathbb{I}_{1}&M\cdot\mathbb{I}_{2}&M\cdot\mathbb{I}_{3}\\
M\cdot\mathbb{I}_{3}&M\cdot\mathbb{I}_{4}&M\cdot\mathbb{I}_{1}&M\cdot\mathbb{I}_{2}\\
M\cdot\mathbb{I}_{2}&M\cdot\mathbb{I}_{3}&M\cdot\mathbb{I}_{4}&M\cdot\mathbb{I}_{1}
\end{bmatrix}.
$$
\end{example}

\begin{proposition}\label{prop:prev}
$\Lambda$ is non-type-preserving.
\end{proposition}
\begin{proof}
Since $M$ is an MDS
matrix, $\Lambda_{\delta,1} \neq 0$. Therefore, $\Lambda_{(n_{w}+2,1):(16,n_{w})}\ne 0$, for each $n_{w}\in\Set{1,\ldots,\delta-2}$, so we can apply Corollary \ref{lem:anontypepres}.
\end{proof}

\begin{corollary}
The mixing layer of $\mathrm{AES}$ is non-type-preserving.
\end{corollary}
\begin{proof}
The result directly follows from Proposition \ref{prop:prev}, anyway we make explicit the algebraic computations in the case of the AES cipher. In \cite{AES} the authors define the mixing layer of $\mathrm{AES}$ using the left matrix action. Since in this paper we are using the right action, we have to consider the transpose of $\Lambda_{\mathrm{AES}}$
$$
(\Lambda_{\mathrm{AES}})^{\mathrm{T}}=
\begin{bmatrix}
\mathbb{I}_{1}\cdot M^{\mathrm{T}}& \mathbb{I}_{4}\cdot M^{\mathrm{T}} &\mathbb{I}_{3}\cdot M^{\mathrm{T}} &\mathbb{I}_{2}\cdot M^{\mathrm{T}}\\
\mathbb{I}_{2}\cdot M^{\mathrm{T}}&\mathbb{I}_{1}\cdot M^{\mathrm{T}}&\mathbb{I}_{4}\cdot M^{\mathrm{T}}&\mathbb{I}_{3}\cdot M^{\mathrm{T}}\\
\mathbb{I}_{3}\cdot M^{\mathrm{T}}&\mathbb{I}_{2}\cdot M^{\mathrm{T}}&\mathbb{I}_{1}\cdot M^{\mathrm{T}}&\mathbb{I}_{4}\cdot M^{\mathrm{T}}\\
\mathbb{I}_{4}\cdot M^{\mathrm{T}}&\mathbb{I}_{3}\cdot M^{\mathrm{T}}&\mathbb{I}_{2}\cdot M^{\mathrm{T}}&\mathbb{I}_{1}\cdot M^{\mathrm{T}}
\end{bmatrix}\in\mathrm{GL}_{16}(\F_{2^8}).
$$
Finally, we note that the coefficient $(16,1)$ of $(\Lambda_{\mathrm{AES}})^{\mathrm{T}}$ is the coefficient $(4,1)$ of 
$$
\mathbb{I}_{4}\cdot M^{\mathrm{T}}=
\begin{bmatrix}
\hex{0} & \hex{0} & \hex{0} & \hex{0}\\
\hex{0} & \hex{0} & \hex{0} & \hex{0}\\
\hex{0} & \hex{0} & \hex{0} & \hex{0}\\
\hex{1} & \hex{1} & \hex{3} & \hex{2}
\end{bmatrix}
$$
that is, $\hex{1}\ne\hex{0}$. Hence $(\Lambda_{\mathrm{AES}})^{\mathrm{T}}_{(n_{w}+2,1):(16,n_{w})}\ne 0$ for each $n_{w}\in\Set{1,\ldots,14}$, so we can apply Corollary \ref{lem:anontypepres}.
\end{proof}

\subsubsection*{Mixing layer of GPig2}
The GPig2 mixing layer  (see \cite{lin-att}) is non-type-preserving. Indeed, it corresponds to the matrix
$$
\setlength{\arraycolsep}{1pt}
   \renewcommand{\arraystretch}{1}
 \Lambda_{\mathrm{GPig2}}=  \tiny{
\left [ 
\begin{array}{cccc|cccc|cccc|cccc}
1&0&0&0& 0&0&0&0& 0&0&0&0& 0&0&0&0\\
0&0&0&0& 1&0&0&0& 0&0&0&0& 0&0&0&0\\
0&0&0&0& 0&0&0&0& 1&0&0&0& 0&0&0&0\\
0&0&0&0& 0&0&0&0& 0&0&0&0& 1&0&0&0\\            
\hline
0&1&0&0& 0&0&0&0& 0&0&0&0& 0&0&0&0\\    
0&0&0&0& 0&1&0&0& 0&0&0&0& 0&0&0&0\\  
0&0&0&0& 0&0&0&0& 0&1&0&0& 0&0&0&0\\  
0&0&0&0& 0&0&0&0& 0&0&0&0& 0&1&0&0\\
\hline                  
0&0&1&0& 0&0&0&0& 0&0&0&0& 0&0&0&0\\   
0&0&0&0& 0&0&1&0& 0&0&0&0& 0&0&0&0\\
0&0&0&0& 0&0&0&0& 0&0&1&0& 0&0&0&0\\
0&0&0&0& 0&0&0&0& 0&0&0&0& 0&0&1&0\\ 
\hline              
0&0&0&1& 0&0&0&0& 0&0&0&0& 0&0&0&0\\
0&0&0&0& 0&0&0&1& 0&0&0&0& 0&0&0&0\\
0&0&0&0& 0&0&0&0& 0&0&0&1& 0&0&0&0\\
0&0&0&0& 0&0&0&0& 0&0&0&0& 0&0&0&1                
\end{array}
\right]},
$$
\noindent  where   $(\Lambda_{\mathrm{GPig2}})_{(3,1):(4,1)}\ne 0$  and $(\Lambda_{\mathrm{GPig2}})_{(4,1):(4,2)}\ne 0$, so we can apply Corollary \ref{lem:anontypepres}.

\section{Applications}\label{sec:app}
We consider an SPNmod cipher with non-type-preserving mixing layer and we prove, under some assumptions, that the group  generated by its round functions is primitive. Similarly, we generalize a GOST-like cipher using a non-type-preserving mixing layer, and thus we obtain the same result under the only hypothesis on the invertibility of the S-Boxes.

\subsection{Primitivity of an SPNmod cipher}\label{sec:security}
In this section we prove that an  SPNmod cipher with invertible S-Boxes and non-type-preserving mixing layer is primitive.
\vspace{2mm}

\noindent Let $$
V= V_1\times V_2\times \ldots\times V_\delta$$ and,  for $1\leq j\leq \delta$, $\dim(V_j) = m$. 
\begin{theorem}\label{secSPNmod}
Let $\mathcal{C}$ be an SPNmod cipher acting on the plaintext space $V$,
  in which a round function has the form
  \begin{equation*}
   \veps_{k}= \gamma\lambda\sigma_k,
  \end{equation*}
  for the round key $k \in V$, where
\begin{itemize}
\item $\gamma\in\sym(V)$ is a non-linear permutation which acts in parallel way on each $V_{j}$, i.e. $\gamma$ is the parallel S-Box
\[(x_1,x_2,\ldots,x_n)\gamma = \left((x_1,\ldots,x_{m})\gamma_1,\ldots,(x_{m(\delta-1)+1},\ldots,x_{n})\gamma_\delta\right),\]
where $\gamma_j\in\sym(V_j)$ and $0\gamma_j\ne 0$.
 \item $\lambda \in \sym(V)$ is a non-type-preserving mixing layer.
 \item  $\sigma_{k}\in \sym(V)$ is the $\boxplus$-translation of $V$ by $k$, i.e. $v\sigma_k=v\boxplus k$, for any $v\in V$.
 \end{itemize}
  \noindent Then $\Gamma_{\infty}=\Gamma_{\infty}(\mathcal{C})$ is primitive.
  
\end{theorem}
\begin{proof}
Recall that $\rho\deq \gamma\lambda$ and that by Lemma \ref{lem:spanTrho} we have $\Gamma_{\infty}=
  \Span{
    \mathcal{T}, \rho
    }$ .  
In order to prove  that $\Gamma_{\infty}$ is primitive,  according to Lemma~\ref{rem1}, 
 we have to show that there are no non-trivial proper subgroup $D$ of $(V, \boxplus)$ and $v\in V$ such that 
$$
D \rho  = v \boxplus D .
$$

Since $0\in D$,  we can take $v=0\rho$, hence it is enough to prove that  if  $D\ne\Set{ 0 }$ is a proper subgroup of $\Z_{2^n}$, then $D\rho\ne 0\rho\boxplus  D$. Clearly, an invertible parallel S-Box  maps  any set having a type to another set having the same type, since each S-box is
a bijection. Hence $D$ and $D\gamma$ share the same type and, by Lemma~\ref{lemma:Wform}, this is the same type as $0\rho\boxplus D$. Therefore $0\rho\boxplus D$ cannot be equal to $D\rho$  if we prove that, for any  non-trivial proper subgroup $D$ of $\Z_{2^n}$, $D\gamma$ and $D\rho=(D\gamma)\lambda$ have different types. Finally, the latter statement follows from Theorem \ref{thm: properties of M}, since by hypothesis $\lambda$ is non-type-preserving. 
\end{proof} 

\begin{remark}
The cipher GPig2 \cite{lin-att} is an example of SPNmod cipher satisfying the hypothesis  of Theorem~\ref{secSPNmod} and so the group generated by its round functions is primitive.
\end{remark}

\subsection{Generalization of the mixing layer of a GOST-like cipher and primitivity}
\label{sec: gost-like}
In this section, we use a known  structure of a block cipher to give an example of a  cipher that is primitive if a non-type-preserving mixing layer is used. In particular, we consider a GOST-like cipher, defined in \cite{GOST_ric}, with a generalized mixing layer using any non-type-preserving matrix instead of a rotation. Then we prove that the group generated by the round functions is primitive if the S-Boxes are invertible. 

\vspace{3mm}

We give the definition of a \emph{generalized GOST-like cipher} and of the corresponding group generated by the round functions, arranging the definition of a GOST-like cipher   given in \cite{GOST_ric} by substituting the rotation by $m\leq s \leq m(\delta-1)$ with any non-type-preserving mixing layer.

The  plaintext
space is $V = V^{1} \times V^{2}$, where $V^{1}, V^{2}$ are two copies
of $\F_2^n$, and the  key   space   $\K$ is another copy of $\F_2^n$.
Clearly $V$ inherits both group structures componentwise from $V^{1}, V^{2}$.\\
\noindent Let us consider 
 \begin{itemize}
 \item $V^{i}$, for $i = 1, 2$, as the Cartesian product
 \begin{equation}\label{eq:directsum}
   V^{i}
   =
   V^{i}_{1} \times \cdots \times V^{i}_{\delta} 
 \end{equation}
  of $\delta  > 1$  spaces $V^{i}_j$,  all of the  same dimension  $m >
 1$;
\item  a non-linear map (parallel S-Box) $\gamma\in\mathrm{Sym}(V^{i})$ which acts in parallel way on each $V^{i}_{j}$,
where $\gamma_j\in\sym(V^{i}_{j})$ and $0\gamma_j\ne 0$;
 \item   a non-type-preserving linear map $\lambda \in \sym(V^{i})$;
 \item $\rho \deq \gamma \lambda \in  \Sym(V^{i})$.
 \end{itemize}

\noindent For $(k_{1}, k_{2}) \in V = V^{1} \times V^{2}$, consider the 
$\boxplus$-translation on $V$ by $(k_{1}, k_{2})$
\begin{equation*}
\begin{array}{rccc}
  \sigma_{(k_{1}, k_{2})} 
  :
  &V^{1} \times V^{2}
  &\longrightarrow &V^{1} \times V^{2}\\
  & (x_{1} , x_{2})
  &\longmapsto
  & (x_{1} \boxplus k_{1}, x_{2} \boxplus k_{2}).
\end{array}
\end{equation*}

We now introduce a formal $2n \times 2n$ matrix, which implements the
Feistel structure,
\begin{equation}\label{eq:Sigma}
  \mathcal{P} = \begin{bmatrix}0 & 1\\ 1 & \rho\end{bmatrix},
\end{equation}
where $0$ and $1$ are $n \times n$ matrices.
This acts (on the right) on $(x_{1}, x_{2}) \in V = V^{1} \times V^{2}$ by
\begin{equation}\label{eq:Sigma-acts}
  (x_{1}, x_{2}) \mathcal{P}
  =
  (x_{2}, x_{1} \oplus x_{2} \rho).
\end{equation}
We are ready to define a round function of a \emph{generalized GOST-like cipher}. Let $\Hc = \K \times \K = V$
be the key space, a round takes the form
\begin{equation}\label{eq:round-like}
  \sigma_{k} \, \mathcal{P} \, \sigma_{h},
\end{equation}
with $k, h \in \Hc$. 

\noindent The corresponding group generated by the round functions will thus be
\begin{equation*}
  \Gamma_{\infty}
  =
  \Span{ 
    \sigma_{k} \, \mathcal{P} \, \sigma_{h}
    :
    k, h \in \Hc
  }.
\end{equation*}
\begin{theorem}\label{secGOSTlike}
Let $\mathcal{C}$ be a generalized GOST-like cipher as defined above. If the parallel S-Box $\gamma$ is a permutation of $V^{i}$, in other words $\gamma\in\sym(V^{i})$, then $\Gamma_{\infty}(\mathcal{C})$ is primitive.
\end{theorem}
\begin{proof}
The proof is the same as the one given  in Section 4 of \cite{GOST_ric}, which uses the Goursat's Lemma \cite{goursat}, until the case $D\rho =  0 \rho\boxplus D$ with $D$ a non-trivial proper subgroup of $\Z_{2^n}$ is reached. Finally, for this case we can proceed as done in the proof of Theorem  \ref{secSPNmod} and apply Theorem \ref{thm: properties of M}.
\end{proof}

\section{Conclusions and open problems}\label{sec:concl}

A key feature of a block cipher is the ability of resisting against known attacks, such as differential, linear and algebraic attacks. In this work we focus on the imprimitivity attack proposed in \cite{Pat}; we approached this problem in the case of block ciphers with addition$\mod 2^n$ as key mixing function. Our main result is the characterization of binary matrices (associated to mixing layers) accordingly to the newly introduced property of being \textit{type-preserving}. Then, we show how non-type-preserving matrices assure resistance against imprimitivity attacks (see Theorems \ref{secSPNmod} and \ref{secGOSTlike}).
\\
The study of primitivity in block ciphers is dependent on the key mixing function. Therefore, it could be interesting to adapt the definition of non-type-preserving mixing layer to other actions of the key. Future directions will be the analyses of $n$-bits block ciphers whose key mixing function is the addition$\mod 2^m$, acting in parallel on disjoint subsets of $m|n$ bits of the state. We remark that the case $m=n$ is the topic of this work, while the case $m=1$ implies that the key mixing function is the addition$\mod 2$ between the key and the state, hence it is already discussed in \cite{CDVS09}.\\
A further work will be to design an instance of the generalized GOST-like cipher, presented in Section \ref{sec: gost-like}, by choosing a non-type-preserving mixing layer, a parallel S-Box and a key-schedule and then to make a more detailed analysis of its security, including the study of classical statistical
attacks. This approach could indeed give new insights on ciphers using addition$\mod 2^n$ as key mixing function.


\paragraph{Acknowledgment}
The authors are grateful to the anonymous referees for their insightful comments and suggestions.




\begin{thebibliography}{00}
\bibitem{SERPENT} R.\,J. Anderson, E. Biham, and L.\,R. Knudsen, {\em
SERPENT: A new block cipher proposal}, Fast Software Encryption, 222--238, Lecture Notes in Comput. Sci. {\bf 1372}, Springer, Berlin (1998).

\bibitem{camelia} K. Aoki,  et al. {\it Camellia: A 128-bit block cipher suitable for multiple platforms-design and analysis}, Selected Areas in Cryptography. 39--56, Lecture Notes in Comput. Sci., {\bf 2012}, Springer, Berlin (2000).

\bibitem{PriPre} R. Aragona, M. Calderini, A. Tortora, and M. Tota,  {\em On the primitivity of PRESENT and other lightweight ciphers}, Journal of Algebra and Its Applications, {\bf 17} (2017), no. 6, 1850115 (16 pages).
\bibitem{ACDVS} R. Aragona, A. Caranti, F. Dalla Volta, and M. Sala, {\em
On the group generated by the round functions of translation based ciphers over arbitrary fields},
Finite Fields Appl. {\bf 25} (2014), 293--305.

\bibitem{GOST_ric} R. Aragona, A. Caranti, and M. Sala, {\it The group generated by the round functions
of a GOST-like cipher}, Ann. Mat. Pura Appl., {\bf 196} (2016), no. 1, 1--17.

\bibitem{bea} A. Bannier, N. Bodin, and E. Filiol, {\em Partition-Based Trapdoor Ciphers}, IACR Cryptology ePrint Archive, Report 2016/493 (2016); available at \url{http://eprint.iacr.org/2016/493}.



\bibitem{PRESENT} A. Bogdanov et al., {\em
PRESENT: An ultra-lightweight block cipher}, CHES '07, 450--466, Lecture Notes in Comput. Sci. {\bf 4727}, Springer, Berlin (2007).

\bibitem{mars} C. Burwick, et al. {\em MARS-a candidate cipher for AES}, NIST AES Proposal {\bf 268} (1998).

\bibitem{cal18} M. Calderini, {\em A note on some algebraic trapdoors for block ciphers}, Adv. Math. Commun. {\bf 12} (2018), no. 3, 515--524.

\bibitem{Ca15} M. Calderini, and M. Sala {\it Elementary abelian regular subgroups as hidden sums for cryptographic trapdoors}, preprint, arXiv:1702.00581 [math.GR] (2017).

%
\bibitem{Cam} P. J. Cameron, {\em Permutation groups}, London Mathematical Society Student Texts {\bf 45}, Cambridge University Press, Cambridge (1999).
%
%
%
\bibitem{ONAN} A. Caranti, F. Dalla  Volta, and M. Sala, {\it An application of the O'Nan-Scott theorem to the group generated by the round functions of an AES-like cipher},  Des. Codes Cryptogr. {\bf 52} (2009), no. 3, 293--301.
%
\bibitem{CDVS09} A. Caranti, F. Dalla Volta, and M. Sala, {\em
On some block ciphers and imprimitive groups},
Appl. Algebra Engrg. Comm. Comput. {\bf 20} (2009), no. 5-6, 339--350.
%
\bibitem{copp} D. Coppersmith and E. Grossman, {\em Generators for certain alternating groups with applications to cryptography},
SIAM J. Appl. Math. {\bf 29} (1975), no. 4, 624--627 .
\bibitem{AES} J. Daemen and V. Rijmen, {\it The design of Rijndael: AES -- the Advanced Encryption Standard}, Information Security and Cryptography, Springer-Verlag, Berlin (2002).

\bibitem{prop-mod} S. M. Dehnavi, A. M. Rishakani, M. M. Shamsabad, H. Maimani, E. Pasha, {\em Cryptographic Properties of Addition Modulo $2^n$}. IACR Cryptology ePrint Archive {\bf 181} (2016).

\bibitem{GOST} V. Dolmatov, {\it GOST 28147�89: encryption, decryption, and message authentication code (MAC) algorithms},
Technical report (2010), \url{http://tools.ietf.org/html/rfc5830}.

\bibitem{goursat} E. Goursat, \emph{Sur les substitutions orthogonales et les divisions r\'eguli\`eres de
l'espace}, Ann. Sci. \' Ecole Norm. Sup. 3(6) (1889), 9--102.

\bibitem{kalinski} Jr. B. S. Kaliski, R. L. Rivest, and A. T. Sherman, {\em Is the Data Encryption Standard a group? (Results of cycling experiments on DES)}, J. Cryptology {\bf 1} (1988), no. 1, 3--36.

\bibitem{alg-att} O. Kazymyrov, R. Oliynykov, H. Raddum, {\em Influence of addition modulo $2^n$ on algebraic attacks}, Cryptogr.  Commun. {\bf8} (2016), no. 2, 277--289.

\bibitem{idea}  X. Lai, J. L. Massey, {\em A proposal for a new block encryption standard},  Advances in cryptology -- EUROCRYPT '90, 389--404, Lecture Notes in Comput. Sci. {\bf 473}, Springer, Berlin (1990).


\bibitem{lin-att} D. Mukhopadhyay,  D. RoyChowdhury. {\em Key Mixing in Block Ciphers through Addition modulo $2^n$}, IACR Cryptology ePrint Archive {\bf 383} (2005).

%
\bibitem{Pat} K.\,G. Paterson, {\em Imprimitive permutation groups and trapdoors in iterated block ciphers}, Fast Software Encryption, 201--214,
Lecture Notes in Comput. Sci. {\bf 1636}, Springer, Berlin (1999).


\bibitem{RC6} R. L. Rivest, M. J. W. Robshaw, R.Sidney, Y. L. Yin, {\em The RC6${}^{TM}$ block cipher}. In First Advanced Encryption Standard (AES) Conference (1998).


\bibitem{SW} R. Sparr and R. Wernsdorf, {\it Group theoretic properties of Rijndael-like ciphers}, Discrete Appl. Math. {\bf 156} (2008), no. 16, 3139--3149.

\bibitem{sha49} C. E. Shannon, {\em Communication theory of secrecy systems}, Bell System Tech. {\bf 28} (1949), 656--715.

\bibitem{sea} F. X. Standaert, G. Piret, N.Gershenfeld, N., J. J. Quisquater, (2006, April). {\em SEA: A scalable encryption algorithm for small embedded applications,} Smart Card Research and Advanced Applications -- CARDIS '06, 222--236, Lecture Notes in Comput. Sci. {\bf 3928}, Springer, Berlin, (2006).

\bibitem{We1} R. Wernsdorf, {\em The round functions of RIJNDAEL generate the alternating group}, Fast Software Encryption 143-148, Lecture Notes in Comput. Sci. {\bf 2365}, Springer, Berlin (2002).
%
\bibitem{We3} R. Wernsdorf, {\em The one-round functions of the DES generate the alternating group}, Advances in cryptology-EUROCRYPT '92, Lecture Notes in Comput. Sci. {\bf 658}, Springer, Berlin (1993).
\bibitem{We2} R. Wernsdorf, {\em The round functions of SERPENT generate the alternating group} (2000); available at
\url{http://csrc.nist.gov/archive/aes/round2/}\linebreak\url{comments/20000512-rwernsdorf.pdf}.

\end{thebibliography}
\end{document}